\documentclass[10pt,a4paper]{article}
\usepackage{mathptmx}       
\usepackage{helvet}         
\usepackage{courier}        
\usepackage{type1cm}        
\usepackage{makeidx}         
\usepackage{graphicx}        
\usepackage{multicol}        
\usepackage{amsmath,amsthm}

\usepackage{amsfonts,graphicx}
\usepackage{amssymb,hyperref}
\usepackage{color,url}
\usepackage{natbib}

\newcommand{\M}{\mathcal{M}}

\DeclareMathOperator*{\argmax}{arg\,max}
\newtheorem{thm}{Theorem}

\newtheorem{lem}[thm]{Lemma}
\newtheorem{rmk}{Remark}
\newtheorem{corollary}{Corollary}
\newtheorem{definition}{Definition}

\begin{document}
	\begin{center}
		\Large{\textbf{Robust Fusion Methods for Structured Big Data}}\\[1.cm]
	\end{center}	\begin{center}
	Catherine Aaron$^a$, Alejandro Cholaquidis$^b$, Ricardo Fraiman$^b$ and Badih Ghattas$^c$
\end{center}
	\normalfont

	\begin{center}
		$^a$ Universit\'e Clermont Auvergne, Campus Universitaire des C\'ezeaux, France.\\
		$^b$ Universidad de la Rep\'ublica, Facultad de Ciencias, Uruguay.\\
		$^c$ Aix Marseille Universit\'e, CNRS, Centrale Marseille, I2M UMR 7373, 13453, Marseille, France.
	\end{center}

\begin{abstract}
	We address one of the important problems in Big Data, namely how to combine estimators from different subsamples by robust fusion procedures, when we are unable to deal with the whole sample.
We propose a general framework based on the classic idea of `divide and conquer'.
In particular we address in some detail the case of a multivariate location and scatter matrix, the covariance operator for functional data, and clustering problems.
\end{abstract}

\section{Introduction}
Big Data has arisen in recent years to deal with problems in several domains, such as social networks,
biochemistry, health care systems,  politics, and retail, among many others. New developments are necessary to address most of the problems in the area. Typically, classical statistical approaches that perform reasonably well for small data sets fail when dealing with huge data sets. To handle these challenges, new mathematical and computational methods are needed.  \\
The challenges posed by Big Data cover a wide range of various problems, and have been recently considered in a huge literature (see, for instance, \cite{wang2016}, \cite{yu2014}, \cite{ah:17} and the references therein). We address one of these problems, namely, how to combine, using robust techniques, estimators obtained from different subsamples in the case where we are computationally unable to deal with the whole sample. In what follows, we will refer to such approaches as robust fusion methods (RFM).

A general  algorithm is proposed, which is, in spirit, related with the well known idea of divide-and-combine. We consider the case where the data belong to finite and infinite dimensional spaces (functional data).

Functional Data Analysis (FDA) has become a central area of statistics in recent years, having gained much momentum from the work of Ramsay in the early 2000s. Since then, both the quantity and the quality of its results have enjoyed a marked growth, while addressing a great diversity of problems. 
FDA faces several specific challenges, most of them  associated with the infinite-dimensional nature of the data. Some recent important and unavoidable references  for FDA are \cite{hk:12}, \cite{fv:06}, \cite{an:17}, as well as the recent surveys \cite{cu:14} and \cite{vz:00}. 

Divide-and-combine (see for instance \cite{aho:74}) is a well known technique for dealing with hugh data-sets. In the FDA setting have, in \cite{song16}, also been considered recently for the linear regression problem involving Lasso, a problem that is not addressed in the present paper, where we focus on a general robust procedure for different problems.  

The consistency and robustness of our method is studied in the general setting of FDA, and we apply the proposed algorithm to some statistical problems in finite and infinite dimensional settings, namely, the location and scatter matrix, clustering, and impartial trimmed $k$-means.
Also, a new robust estimator of the covariance operator is proposed.

We start by describing one of the simplest problems in this area as a toy example.
Suppose we are interested in the median of a huge set of iid  random variables $\{X_1, \ldots, X_n\}$ with common density $f_X$, and we split the sample into $m$ subsamples of size $l$, so that $n=ml$.
We calculate the median of each subsample and obtain $m$ random variables $Y_1, \ldots, Y_m$.
Then we take the median of the set $Y_1, \ldots, Y_m$, i.e. we consider the well known median of medians, which, in this case, will be our RFM estimator.
It is clear that it does not coincide with the median of the whole original sample $\{X_1, \ldots, X_n\}$, but it will be close.
What else can we say about this estimator regarding its efficiency and robustness?

In this particular case, the RFM estimator is nothing but the median of $m$ iid random variables, but now with a different distribution, given by the distribution of the median of
$l$ random variables with density $f_X$.
Suppose for simplicity that $l=2k+1$.
Then, the density of the random variables $Y_i$ is given by 
\begin{equation}\label{densitymedian}
g_Y(t) = \frac{(2k+1)!}{(k!)^2} F_X(t)^k (1-F_X(t))^kf_X(t).
\end{equation}

On the one hand, if $f_X(F_X^{-1}(0.5))\neq 0$, the empirical median $\hat{\theta}=$\break $med(X_1, \ldots, X_n)$ 
behaves, asymptotically, like the normal distribution centred at the true median $\theta$ with variance $\mathbb{V}(\hat{\theta})=1/(4 n f_X(\theta)^2)$.
On the other hand, $\tilde{\theta}^{RFM}$, the median of medians, behaves asymptotically like the normal distribution centred at 
$\theta$ with variance $\mathbb{V}(\tilde{\theta}^{RFM})=1/(4 m g_Y(\theta)^2)$, where $g_Y(\theta)= (1/2)^{2k}(2k+1)!/(k!^2)f_X(\theta) \sim\sqrt{2k/\pi}$. 
So we can explicitly calculate the asymptotic relative loss of efficiency, i.e. $\lim_{n\to \infty}\mathbb{V}(\hat{\theta})/\mathbb{V}(\tilde{\theta}^{RFM}) = 2/\pi$.\\
In Section \ref{general} we generalize  this RFM idea and study its consistency, robustness, breakdown point, and efficiency. Section \ref{application} shows how the RFM may be applied to multivariate location and scatter matrix estimation, covariance operator estimation for functional data, and robust clustering. The last section provides some simulation results for these problems.

\section{A general setup for RFM.} \label{general}

We start by introducing a general framework for RFM.
The idea is quite simple: given a sample $\{X_1,\ldots, X_n\}$ of iid random elements in a metric space $E$ (for instance $E=\mathbb{R}^d$) and a statistical problem, (such as multivariate location, covariance operators, linear regression, or principal components,  among many others), we split the sample into $m$  subsamples of equal size.
For each subsample we compute a robust solution for the statistical problem considered.
The solution given by RFM corresponds to the deepest point among the $m$ solutions (in terms of the appropriate norm associated to the problem) obtained from the subsamples.
In order to introduce the notion of depth, we will use throughout this paper the following notation. Let  $X$ be a random variable taking values in some Banach space $(E,\|\cdot\|)$, with probability distribution $P_X$, and let $x\in E$.
The depth of $x$ with respect to $P_X$ is defined as follows:

\begin{equation}\label{depth}
D(x,P_X)= 1- \left\|\mathbb E_{P_X} \left( \frac{X-x}{\Vert X - x\Vert}\right)\right\|.
\end{equation}
It was introduced by \cite{ch:96},  formulated (in a different way) by \cite{vz:00},  and extended to a very general setup by \cite{chch:14}.

Given a sample $\{X_1,\dots,X_n\}$, let us write $P_n$ for the empirical measure. The empirical version of \eqref{depth} is
\begin{equation} \label{empdepth}
D(x,P_n)=1-\left\|\mathbb{E}_{P_n}\left( \frac{X-x}{\Vert X - x\Vert}\right)\right\|=1-\frac{1}{n}\left\|\sum_{i=1}^n \frac{X_i-x}{\Vert X_i - x\Vert}\right\|.
\end{equation}

Although we suggest using the depth function, for some statistical problems this is unsuitable, for instance in clustering.
In such cases, the deepest point may be replaced by other robust estimators, as we will show in Section \ref{robclu}.
We summarize our approach in Table \ref{RFM-alg} for a general framework of parameter estimation.
This may be easily applied to any situation where robust estimators exist or can be designed.

\begin{table}[h!]
	\begin{tabular}{|l|}
		\hline
		$\{X_1,\ldots, X_n\}$ iid random elements in a Banach space $E$.\\
		$\theta_0$ a parameter to estimate \\ \hline
		a)  split the sample into $m$ subsamples with $n=ml$ \\
		\hspace{1cm}  $\{X_1,\ldots, X_l\},\{X_{l+1},\ldots, X_{2l}\},\ldots,$ $\{X_{(m-1)l+1},\ldots, X_{lm}\}$.
		\\
		b) Compute a robust estimate of $\theta_0$ on each subsample, obtaining $\hat{\theta}_1,\ldots,\hat{\theta}_m$.
		\\
		c) Compute the final estimate $\tilde{\theta}^{RFM}$ by RFM combining $\hat{\theta}_1,\ldots,\hat{\theta}_m$ \\ by a robust approach.
		\\
		For instance, $\tilde{\theta}^{RFM}$ can be the deepest point, or the average of $40\%$\\
		of the deepest points among the $\hat{\theta}_1,\ldots,\hat{\theta}_m$.\\ \hline
	\end{tabular} 
	\caption{Parameter estimation using RFM}
	\label{RFM-alg}
\end{table}

We will address the consistency, efficiency, robustness, and computational time of the RFM proposals.

\newpage
\subsection{Consistency, robustness and breakdown point of the RFM} \label{consist}

We start by proving that, given a sample $\{X_1,\dots,X_n\}$ of a random element $X$, its deepest point (i.e. the value that maximizes \eqref{empdepth}) converges almost surely to the value that maximizes \eqref{depth}. Although similar results has already been obtained (see for instance \cite{chch:14}), we will need it when $P_n$ is not necessarily the empirical measure associated to a sample, but any measure converging weakly to a probability distribution $P$.
We will need the following assumption. \\

\textbf{H1} A probability measure $P$ defined on a separable Hilbert space $\mathcal{H}$ fulfils H1 if $P(\partial B(y,r))=0$ for all $r>0$ and $y\in \mathcal{H}$, where $\partial A$ stands for the boundary of a set $A\subset \mathcal{H}$. \\
 Observe that H1 is fulfilled if the random variables $\|X-y\|$ are absolutely continuous, for all $y\in \mathcal{H}$, where $X$ is a random variable with distribution $P$.

\begin{thm} \label{cons} Let $\{X_n\}_n$ be a sequence of random elements with common distribution $P_n$, defined in a separable Hilbert space $(\mathcal{H},\|\cdot\|)$. Let $P$ be a probability distribution fulfilling H1. Assume that $P_n\rightarrow P$ weakly, and $\|\mathbb{E}_P((X-x)/\|X-x\|)\|$ has a unique minimum. Then 
	\begin{equation} \label{thcons} 
	\argmax_x \ D(x,P_n)\rightarrow \argmax_x \ D(x,P)\quad  a.s., \text{ as }n\rightarrow \infty.
	\end{equation}
\end{thm}

In order to prove \eqref{thcons} we will use the following fundamental result proved in \cite{bil67} (which still holds when $\mathcal{H}$ is a separable Banach space), see theorem 1 and example 3. \\
 
\textbf{Theorem (Billingsley and Tops{\o}e}). Suppose $S\subset \mathcal{H}$ and let $\mathcal{B}(S,\mathcal{H})$ be the class of all bounded measurable functions mapping $S$ into $\mathcal{H}$. Suppose $\mathcal{F}\subset \mathcal{B}(S,\mathcal{H})$ is a subclass of functions. Then 
\begin{equation} \label{bil1}
\sup_{f\in \mathcal{F}} \left\| \int fdP_n-\int fdP\right\|\rightarrow 0,
\end{equation}
for every sequence $P_n$ that converges weakly to $P$ if, and only if,
$$\sup\{\|f(z)-f(t)\|:f\in \mathcal{F},z,t\in S\}\leq \infty,$$
and for all $\epsilon>0$,
\begin{equation}\label{biltop}
\lim_{\delta\rightarrow 0} \sup_{f\in \mathcal{F}} P(\{x: \omega_f(B(x,\delta))\geq \epsilon\})=0,
\end{equation}
where $\omega_f(A)=\sup\{|f(x)-f(y)|:x,y\in A\}$ and $B(x,\delta)$ is the open ball of radii $\delta>0$. \\

\textit{Proof of Theorem \ref{cons}}.\

Consider $S=\mathcal{H}$ and $\mathcal{F}$ the subclass of functions $\{f_y\}_{y\in \mathcal{H}}$ where $f_y(z)=(z-y)/\|z-y\|$. Then, 
$\sup\{\|f_y(z)-f_y(t)\|:y,x,z\in \mathcal{H}\}\leq 2$. Let $2\sqrt{\delta}<\epsilon$. Then, for all $y$,
\begin{multline*}
\{x: \omega_{f_y}(B(x,\delta))>\epsilon\}=\{x\in B(y,\sqrt{\delta}): \omega_{f_y}(B(x,\delta))>\epsilon\}\cup\\
 \{x\notin B(y,\sqrt{\delta}): \omega_{f_y}(B(x,\delta))>\epsilon\}.
\end{multline*}

 Observe that  $\omega_{f_y}(B(x,\delta))= 2\delta/\|x-y\|$ if $\|x-y\|>\delta$, and so if $x\notin B(y,\sqrt{\delta})$, then 
$\omega_{f_y}(B(x,\delta))\leq 2\sqrt{\delta}<\epsilon$, and so $\{x\notin B(y,\sqrt{\delta}): \omega_{f_y}(B(x,\delta))>\epsilon\}=\emptyset$.
Lastly we get that for all $y$, 
$$\{x: \omega_{f_y}(B(x,\delta))>\epsilon\}=\{x\in B(y,\sqrt{\delta}): \omega_{f_y}(B(x,\delta))>\epsilon\}\subset B(y,\sqrt{\delta}).$$

Now, since $P(\partial B(y,\sqrt{\delta}))=0$ we have that $\mathbb{I}_{B(y_k,\sqrt{\delta})}(x)\rightarrow \mathbb{I}_{B(y,\sqrt{\delta})}(x)$ a.s. w.r.t. $P$,  whenever $y_k\rightarrow y$ for every $y$, and the dominated convergence theorem implies that $P(B(y_k,\sqrt{\delta}))\rightarrow P(B(y,\sqrt{\delta}))$. This entails that $P(B(y,\sqrt{\delta}))$ is a continuous function of $y$, so its maximum in a compact set, is attained. Let $\epsilon>0$ and $K_\epsilon$ be a compact set such that $P((K_\epsilon\ominus B(0,1))^c)<\epsilon$ where $K_\epsilon\ominus B(0,1)=\{z\in K_\epsilon: d(z,K_\epsilon^c)>1\}$. Denote by $y_{\epsilon,\delta}=\argmax_{y\in K_\epsilon} P(B(y,\sqrt{\delta}))$, let us prove that for all fixed $\epsilon>0$, $P(B(y_{\epsilon,\delta},\sqrt{\delta}))\rightarrow 0$ as $\delta\rightarrow 0$. If this is not the case there exists $\eta>0$, $y_n\in K_\epsilon$ and $\delta_n\rightarrow 0$ such that $P(B(y_{n},\sqrt{\delta_n}))>\eta$ for all $n$. Since $K_\epsilon$ is compact we can assume that $y_n\rightarrow y$ for some $y\in K_\epsilon$ (by considering a subsequence). From $P(\partial B(x,r))=0$ for all $x$, it follows that $P(\{y\})=0$ (indeed, consider $x$ and $r>0$ such that $y\in \partial B(x,r)$). Let us define $\rho_n=\max_{j\geq n} (\sqrt{\delta_j}+\|y-y_j\|)$ and $B_n=B(y,\rho_n)$, then  
$P(B_n)\geq \eta$ and $B_1\supseteq B_2 \supseteq\ldots \supseteq B_n \dots$. Finally, $0=P(\{y\})=\lim P(B_n)$, which contradict that $P(B(y_n,\sqrt{\delta_n}))>\eta$.  Now for all $\delta<1$,
\begin{align*}
\sup_{y} P(B(y,\sqrt{\delta}))\leq & \max\Big\{\sup_{y\in K_\epsilon} P\big(B(y,\sqrt{\delta})\big),  P\big((K_\epsilon\ominus B(0,1))^c\big)\Big\},
\end{align*}
therefore $\sup_y P(B(y,\sqrt{\delta}))\leq \max\{P(B(y_{\epsilon,\delta},\sqrt{\delta})),\epsilon\}<\epsilon$ for $\delta$ small enough, showing that \eqref{biltop} holds.
Lastly \eqref{thcons} is a consequence of the uniform convergence of $D(x,P_n)$ to $D(x,P)$ and the argmax argument.\\

The following corollary states the consistency of the RFM explained in Table \ref{RFM-alg} when the sample $X_1,\dots,X_n$ is distributed
as a random variable $X$ with a distribution $P_0$ fulfilling H1.

\begin{corollary} \label{cor0} 
		Assume that $P_0$ fulfils H1 and there exists a unique $\theta_0$ such that, for all $l$,
		\begin{equation*}
		\mathbb E_{P_0}\left(\frac{\hat{\theta}_1 - \theta_0}{\Vert \hat{\theta}_1 - \theta_0\Vert}\right) = 0.
		\end{equation*}
		Then, under $P_0$, 	$\tilde{\theta}^{RFM}\rightarrow \theta_0$ a.s., as $m\rightarrow \infty$. 
\end{corollary}

Recall that a sequence of estimators $\{\hat{\theta}_n\}_n$ is qualitatively robust at a probability distribution $P$ if for all $\epsilon>0$ there exists $\delta>0$, for all probability distribution $Q$, $\Pi(P,Q)<\delta\Rightarrow \Pi(\mathcal{L}_P(\hat{\theta}_n),\mathcal{L}_Q(\hat{\theta}_n))<\epsilon$ (see \cite{hampel}), where $\Pi$ denotes the Prokhorov distance and $\mathcal{L}_F(\hat{\theta}_n)$ denotes the probability distribution of $\hat{\theta}_n$ under $F$. As $\Pi$ metrizes the weak convergence we have the following corollary.

\begin{corollary}  Robustness of RFM estimators. Under the hypotheses of Corollary \ref{cor0}, $\tilde{\theta}^{RFM}$ is qualitatively robust.

\end{corollary}

\begin{rmk} Qualitative robustness ensures the good behaviour of the estimator in a neighbourhood of $P_0$. However, there are some estimators that still converge to $\theta_0$ even if $P$ is far from $P_0$. For instance ``the shorth", defined as the average of the observations lying on the shortest interval containing half of the data, has this property. 
	Indeed, consider the case where $P_0 = U(-1,1)$, $P_1=U(3,4)$, and $P=(1-\alpha) P_0 + \alpha P_1$,  for any $\alpha < 0.5$. This is also the case for the impartial trimmed estimators, the minimum volume ellipsoid, and the redescendent (with compact support) $M$-estimators (see subsection \ref{Mest}). If the estimators for each subsample have this property, the RFM estimator will inherit it.
\end{rmk}

\subsection{Efficiency of the fusion of $M$-estimators} \label{Mest}

In this section we obtain the asymptotic variance of the RFM method, for the special case of $M$-estimators.
Recall that an $M$-estimator $T$ can be defined (see section 3.2 in \cite{hu:09}) by the implicit functional equation
$\int \psi(x;T(F))F(dx)=0$,
where $\psi(x;\theta)=(\partial /\partial \theta)\rho(x;\theta)$ and $F$ stands for the true underlying common distribution of the observations.
For instance, the Maximum Likelihood estimator is obtained with $\rho(x;\theta)=-\log(f(x,\theta))$.
The estimator $T_n$ is given by the empirical version of $T$, based on a sample $\{X_1,\dots,X_n\}$.
It is well known that $\sqrt{n}(T_n-T(F))$ is asymptotically normal with mean 0 and variance $A(F,T)$ given by the integral of the square of the influence curve, i.e.
$A(F,T)=\int IC(x;F,T)^2F(dx)$, where the influence curve, $IC$, is 
$$IC(x;F,T)=\frac{\psi(x;T(F))}{-\int (\partial/\partial \theta)\psi (x;T(F))F(dx)}.$$
For the location problem (i.e. $\int\psi(x-\mu_0)F(dx)=0$), we get $IC(x,F,T)= -\psi(x-\mu_0)/\int \psi'(x-\mu_0)F(dx)$.
The asymptotic efficiency of $T_n$ is defined as 
${\rm Eff}(T_n)=\sigma_{ML}^2/A(F,T)$, where $\sigma_{ML}^2$ is the asymptotic variance of the maximum likelihood estimator.
Then the asymptotic variance of an $M$-estimator built from a sample $T_n^1,\dots,T_n^m$ of $m$ $M$-estimators of $T$ can be calculated easily. The strong consistency of the $M$-estimators under the model (see \cite{hu67}) entails  that $\tilde{\theta}^{RFM}$ built from $M$-estimators is consistent (see Corollary \ref{cor0}) whenever the empirical version of the implicit functional equation has an unique solution. 

 The choice of $m$ and $l$ has an impact on the robustness of the estimator and on the computation time. Indeed, if the computation time of each $\hat{\theta}_i=\mathcal{O}(l^a)$ and the computation time of the fusion step is $\mathcal{O}(m^b)$, then the optimal choice (if $b>1$) is $l=\mathcal{O}(n^{(b-1)/(a+b-1)})$.  

\subsection{Breakdown point for the RFM}

Following \cite{do:82} we consider the finite-sample breakdown point, introduced by Donoho.  Intuitively the breakdown point corresponds to the maximum percentage of outliers (located at the worst possible positions) we can have in a sample before the estimate breaks in the sense that it can be arbitrarily large (or close to the boundary of the parameter space).
\\
\begin{definition} Let $\mathbf{x}=\{x_1,\dots,x_n\}$ be a data-set, $\theta$ an unknown parameter lying in a metric space $\Theta$, and $\hat{\theta}_n=\hat{\theta}_n(\mathbf{x})$ an estimate based on $\mathbf{x}$.
Let $\mathcal{X}_p$ be the set of all data-sets $\mathbf{y}$ of size $n$ having $n-p$ elements in common with $\mathbf{x}$:
	$$\mathcal{X}_p=\{\mathbf{y}:card(\mathbf{y})=n, \ card(\mathbf{x}\cap \mathbf{y})=n-p\}.$$
	Then the breakdown point of $\hat{\theta}_n$ at $\mathbf{x}$ is $\epsilon_n^*(\hat{\theta}_n,\mathbf{x})=p^*/n,$
	where 
		$p^*=\max\{p\geq 0; \forall \mathbf{y}\in \mathcal{X}_p, \ \hat{\theta}_n(\mathbf{y}) \mbox{ is bounded and also bounded away from the}$\\ $\text{ boundary} \ \partial \Theta, \text{ if } \partial \Theta\neq \emptyset  \}.$

\end{definition}

To analyse the breakdown point of the RFM, we consider the case where the breakdown point of the robust estimators is 0.5 (high breakdown point estimators).

For each observation $X_i$ from the sample, let $B_i=1$ if $X_i$ is an outlier and $0$ otherwise.
Assume that the variables $B_i$ are iid following a Bernoulli distribution with parameter $p$ and let $S_j=\sum_{s=1}^l B_{(j-1)l+s}$ be the number of outliers in the subsample $j$, for $j=1, \ldots, m$.
The RFM estimator will breakdown if and only if there are more than $m/2$ cases where $S_j$ is greater than $k$ (recall that $l=2k+1$).

To take a glance of the behaviour of the breakdown point, we performed $5000$ replications where we generated $n=30000$ binomial random variables with parameter $p$.  We split each of the samples of size $30000$  randomly into $m$ subsamples. Next we calculated the number of its subsamples which contained more than $k$ 1's (outliers). In Table \ref{tab1} we report the average number of times (over the 5000 replications) that this number was greater than $m/2$, for different values of $p$ and $m$. The best result is obtained for $m=5$.

\begin{table}[h]
	\footnotesize
	\begin{center}
		\begin{tabular}{|c|c|c|c|c|}
			\hline
			$m$	& $p=0.45$  & $p=0.49$& $p=0.495$  & $p=0.499$\\
			\hline              
			5  & 0  & 0.0020    & 0.0820 & 0.3892  \\
			10 & 0  & 0.0088    & 0.1564 & 0.5352  \\
			30 & 0  & 0.0052    & 0.1426 & 0.5186 \\  
			50 & 0	& 0.0080    & 0.1598 & 0.5412\\    
			100& 0	& 0.0192    & 0.2162 & 0.6084   \\  
			150& 0	& 0.0278    & 0.2728 & 0.6780 \\  
			
			\hline
		\end{tabular}
	\end{center}
	\caption{Average (over 5000 replications) of estimator breakdowns for different values of $m$ and $p$ and fixed $n=30000$; $p$ is the proportion of outliers.}
	\label{tab1}
\end{table}

\section{Some applications of RFM \label{application}}
In this section we will show how RFM may be used to tackle three classic statistical problems for large samples: estimating the multivariate location and scatter matrix, estimating the covariance operator for functional data, and clustering.
For each problem we show how to apply our approach, given in Table \ref{RFM-alg}.
Solutions for many other problems may be derived from these cases (Principal Components, for example, both for non-functional and functional data).

\subsection{Robust fusion for location and scatter matrix in finite dimensional spaces}
 Given an iid random sample $\{X_1,\dots,X_n\}$ in $\mathbb{R}^d$, we consider the location and scatter matrix estimation problem.

To perform RFM we only need to make explicit the estimators used for each of the $m$ subsamples, and the depth function in the fusion stage.
For the location parameters, we propose to use simple robust estimates, denoted by $\hat{\theta}_1,\dots,\hat{\theta}_m$  (see for instance \cite{mmy:06}). 

For the depth function we propose to use the empirical version of
\eqref{depth}, replacing $P_X$ by the empirical distribution $P_m$ of $\{\hat{\theta}_1,\dots,\hat{\theta}_m\}$, 
\begin{equation}\label{dephsloc}
D(\theta, P_m) = 1 - \left\| \frac{1}{m}\sum_{j=1}^m \frac{\hat{\theta}_{j} - \theta}{\|\hat{\theta}_{j}- \theta\|}\right\|,
\end{equation}
where $\theta \in \mathbb{R}^d$, and $\| \cdot \|$ is the Euclidean distance.
Equivalently, for the scatter matrix we use the depth function
\begin{equation}\label{dephscat}
D(\Sigma,P_m) = 1 - \left\| \frac{1}{m}\sum_{j=1}^m \frac{\hat{\Sigma}_{j} - \Sigma}{\|\hat{\Sigma}_{j} - \Sigma\|}\right\|,
\end{equation}
where $\hat{\Sigma}_1,\dots,\hat{\Sigma}_m$ are robust estimators of the scatter matrix, the norm is $\|\Sigma\| = \max_{1\le i \le d} \sum_{j=1}^d | \Sigma_{ij}|$. $P_m$ denotes the empirical distribution of $\{\hat{\Sigma}_1,\dots,\hat{\Sigma}_m\}$. 
A simulation study is presented in Section \ref{sec:sim}.

\subsection{Robust fusion for the covariance operator} \label{covop}

The estimation of the covariance operator of a stochastic process is a very important topic in FDA, which helps to understand the fluctuations of a random element, 
as well as to derive the principal functional components from its spectrum.
Several robust and non-robust estimators have been proposed, see for instance \cite{chch:14} 
and the references therein.
In order to perform RFM, we introduce a new robust estimator to use for each of the $m$ subsamples, 
which can be implemented using parallel computing.
It is based on the notion of impartial trimming in the Hilbert--Schmidt space where the covariance operators are defined.
It was introduced in \cite{go:91} and has been shown to be a very successful tool in robust estimation.
Next, the RFM estimator is defined as the deepest point among the $m$ estimators (`impartial trimmed means') corresponding to each subsample.\\
To better understand the construction of our new estimator, we will first recall the general framework used for the estimation of covariance operators.

\subsubsection{A general framework for the  estimation of covariance operators}

Let $E=L^2(I)$, where $I$ is a finite interval in $\mathbb R$, and $X, X_1, \ldots X_n, \ldots$ be iid random elements taking values in $E$. Assume that $\mathbb E(X(t)^2) < \infty$ for all $t\in I$, and $\int_I \int _I \rho^2(s,t) dsdt < \infty$, so that the covariance function, given by
$
\rho(s,t)=\mathbb E((X(t)-\mu(t))(X(s)-\mu(s))) \ \ \mbox{where} \ \mathbb E(X(t))=\mu(t)
$, is well defined.
For notational simplicity we  assume that $\mu(t)=0,\forall t\in I$. 
Under these conditions, the covariance operator, given by
\begin{equation}
\Gamma_0 (f) (t)= \mathbb E(\langle X,f\rangle X(t))= \int_I \rho(s,t)f(s)ds, \; \; f \in E,
\end{equation}
is diagonalizable, with non-negative eigenvalues $\lambda_i$ such that $ \sum_i \lambda^2_i < \infty$.  Moreover $\Gamma_0$ belongs to the Hilbert--Schmidt space $HS(E)$ of linear operators with norm and inner product given by
\begin{equation}\label{norm}
\Vert \Gamma \Vert^2 _{HS} = \sum_{k=1}^{\infty} \Vert \Gamma(e_k) \Vert^2 < \infty, \ \ \langle \Gamma_1,\Gamma_2\rangle_{HS} = \sum_{k=1}^{\infty} \langle \Gamma_1(e_k), \Gamma_2(e_k)\rangle,
\end{equation}
respectively, where $\{e_k: k \geq 1\}$ is any orthonormal basis of $E$, and $\Gamma,\Gamma_1, \Gamma_2 \in  HS(E)$. 
In particular, $\Vert \Gamma_0 \Vert_{HS}^2 = \sum_{i=1}^{\infty} \lambda_i^2$.
Given an iid sample $\{X_1, \ldots,X_n\}$, we define the Hilbert--Schmidt operators of rank-one, $W_i: E \rightarrow E$, as
$$W_i(f) = \langle X_i, f \rangle X_i(.),  \hspace{2mm}  i=1, \ldots n.
$$
Let $\phi_i = X_i/\Vert X_i\Vert$, then $W_i(\phi_i ) = \Vert X_i\Vert^2 \phi_i=: \eta_i \phi_i.$

The standard estimator of $\Gamma_0$ is just the average of these operators, i.e. $\hat \Gamma_n = \frac{1}{n} \sum_{i=1}^n W_i$, which is a consistent estimator of $\Gamma_0$ by the Law of Large Numbers in the space $HS(E)$.
We replace this average by a trimmed version in the space $HS(E)$.

\subsubsection{A new robust estimator for the covariance operator}

Our proposal is to consider an impartial trimmed estimator as a resistant estimator. The notion of impartial trimming was introduced in \cite{go:91}, and the functional data setting was considered in \cite{caf:06}, from where one can can obtain the asymptotic theory for our setting. The construction of our estimator needs an explicit expression of the distances $\Vert W_i - W_j \Vert$, $1 \leq i \leq j \leq n$, which we will derive using the following lemma.

\begin{lem}  We have that 
	\begin{equation} \label{eqlem1}
	d_{ij}^2:=\Vert W_i - W_j \Vert_{HS}^2 =  \Vert X_i\Vert^4 + \Vert X_j\Vert^4 - 2  \langle X_i, X_j\rangle^2 \quad  \text{ for } 1 \leq i \leq j \leq n.
	\end{equation}
\end{lem}
\begin{proof}	
	Let us write
	\begin{align*}
	\langle W_i  -W_j, W_i  -W_j\rangle_{HS} =&\  \langle W_i  , W_i\rangle_{HS} + \langle W_j , W_j\rangle_{HS} -2\langle W_i , W_j\rangle_{HS} \\
	=&\ \eta_i^2 +  \eta_j^2 -2 \sum_{k=1}^{\infty} \langle W_i(e_k), W_j(e_k)\rangle\\
	=&\  \eta_i^2 +  \eta_j^2 -2 \sum_{k=1}^{\infty}\big\langle \langle X_i,e_k\rangle X_i, \langle X_j,e_k\rangle X_j\big\rangle\\
	=&\ \eta_i^2 +  \eta_j^2 -2\langle X_i,X_j\rangle \sum_{k=1}^{\infty}\langle X_i,e_k\rangle \langle X_k,e_k\rangle.
	\end{align*}
	Now Eq. \eqref{eqlem1} follows from the identity
	$$\sum_{k=1}^{\infty}\langle X_i,e_k\rangle \langle X_k,e_k\rangle =\langle X_i,X_j\rangle.$$
\end{proof}

Given the sample, which we have assumed with mean zero for notational simplicity, and $0< \alpha< 1$, we provide a simple algorithm to calculate an approximate impartial trimmed mean estimator of  the covariance operator which is strongly consistent.\\

\textbf{STEP 1:} Calculate $d_{ij}=\Vert W_i - W_j \Vert_{HS}$,  $1 \leq i \leq j \leq n$, using Lemma 1.\\

\textbf{STEP 2:}  Let $r=\lfloor(1-\alpha)n\rfloor+1$.
For each $i=1, \ldots n$, consider the set of indices $I_i \subset \{1, \ldots ,n\}$ corresponding to the $r$ nearest neighbours of $W_i$ among $\{W_1, \ldots, W_n\}$, and
the order statistic of the vector $(d_{i1}, \ldots , d_{in})$, $d_i^{(1)}\leq \ldots \leq d_i^{(n)}$.\\

\textbf{STEP 3:} Let $\gamma = \text{argmin} \{ d_1^{(r)}, \ldots, d_n^{(r)} \}$. \\

\textbf{STEP 4:} The impartial trimmed mean estimator of $\Gamma_0$ is given by the average of the $r$ nearest neighbours of $W_{\gamma}$ among  $\{ W_1, \ldots, W_n\}$, i.e the average of the rank-one operators $W_i$ such that $i\in I_\gamma$.
The covariance function is then estimated by $\hat\rho(s,t) = \frac{1}{r} \sum_{j \in I_\gamma }X_j(s)X_j(t)$. Observe that Steps 1 and 2 of the algorithm can be performed using parallel computing.

The final estimator given by the RFM may be obtained by taking the deepest point (or the average of the $40\%$ deepest points) among the $m$ estimators obtained from the algorithm above. The norm used for the depth function  in this case is the functional analogue of \eqref{dephscat}.

\subsection{Robust fusion for cluster analysis} \label{robclu}
In this section we describe a robust fusion method for clustering. Our approach is based on the use of impartial trimmed $k$--means (ITkM, see \cite{cagm:97}) in two steps. In the first one we apply ITkM with a given trimming level $\alpha_1$ to each of the $m$ subsamples, and obtain $m$ sets of $k$ centres $\hat{\mathcal M_1}, \ldots, \hat{\mathcal M}_m$. In the second step we apply ITkM with a trimming level $\alpha_2$ to the set $\cup_{i=1}^m \hat{\mathcal M_i}$,  as suggested in \cite{cagm:97} (Section 5.1). We start by describing briefly ITkM. 

\subsubsection{Impartial trimmed $k$-means} Given a sample $\{X_1, \ldots, X_n\} \subset \mathbb R^d$, a trimming level $0< \alpha <1$, and the number of clusters $k$, ITkM looks for a set $\{\hat m_1, \ldots, \hat m_k\} \subset \mathbb R^d$ and a partition of the space $C_0, C_1, \ldots, C_k$ that minimizes the loss function
$$
\frac{1}{n - [n\alpha]} \sum_{j=1}^k \sum_{X_i \in C_j} \Vert X_i - \hat m_j \Vert ^2.
$$
Here, $C_0$ is the set of trimmed data (with cardinality $\lfloor n \alpha\rfloor$).
Let $X \in \mathbb R^d$ be a random vector with distribution $P_X$, the number of clusters $k$, and a trimming proportion $0<\alpha <  1$. 
\begin{itemize}
	\item For every $k$-set $\mathcal{M} = \{m_1, \ldots, m_k \}$, with $m_j \in \mathbb R^d$ for all $j=1, \ldots, k$, and $x \in \mathbb R^d$, we define
	$$
	d(x, \mathcal{M}) := \min\big\{\Vert x - m_1\Vert, \ldots, \Vert x - m_k \Vert\big\}.
	$$
	\item The set of trimming functions for $P_X$ at level $\alpha$ is defined by
	$$
	\tau_{\alpha}(P_X) = \Big\{ \tau: \mathbb R^d \to [0,1], \ \mbox{measurable, fulfilling} \ \ \int \tau(x) dP_X(x) \geq 1 - \alpha\Big\}.
	$$
	The functions in $\tau_{\alpha}(P_X)$ are a natural generalization of the indicator functions $\mathbf {1}_A$ with $P_X(A) = 1 - \alpha$.
	
	\item For each pair  $(\tau, \M)$ such that $\tau \in \tau_{\alpha}(P_X)$ and $\M\subset \mathbb{R}^{d}$ with $\#\M=k$,  let us consider the function
	$$
	\mathcal{V}(\tau, \M, P_X) = \frac{1}{\int \tau(x) dP_X(x)} \int \tau(x) d^2(x,\M) dP_X(x).
	$$
	Lastly, we define
	\begin{equation} \label{func}
	\mathcal{V}(P_X)= \inf_{\tau\in \tau_{\alpha}(P_X)} \inf_{\substack{\M\subset\mathbb{R}^{d}\\ \#\M=k}} \mathcal{V}(\tau,\M,P_X).
	\end{equation}
\end{itemize} 
Corollary 3.2 in \cite{cagm:97} proves that there exists a pair (not necessarily unique) $(\tau^*, \M^*)$ attaining the value $\mathcal{V}(P_X)$.
Moreover, if $P_X$ is absolutely continuous w.r.t. Lebesgue measure,  $\tau^* = \mathbf 1 _{B(\M^*, r^*)}$ with   $r^*=r(\alpha, \M^*)=\inf\{r\geq 0: P_X(B(\M^*,r)) \geq 1-\alpha\}$ and $B(\M^*,r)=\{x\in \mathbb{R}^d:d(x,\M^*)\leq r\}$.

Let us denote by $P_n$ the empirical distribution based on the sample. Theorem 3.6 in \cite{cagm:97} proves that if $P_X$ is absolutely continuous w.r.t. Lebesgue measure and there exists a unique pair $(\tau^*, \M^*)$ solving \eqref{func}, then $\mathcal{V}(P_n)\rightarrow \mathcal{V}(P_X)$ a.s.
Moreover, if ${\hat{\M}}$ is any sequence of empirical trimmed $k$-means, then $d_H(\hat{\M},\M^*)\rightarrow 0$ a.s., where $d_H$ denotes the Hausdorff distance.

It is clear that in this case $\hat{\tau}_n=\mathbf{1}_{B(\hat{\M}, \hat{r})}\rightarrow \tau^*$ $P_X$ a.s., where $\hat{r}=\inf\{r\geq 0: P_n(B(\hat{\M},r))\geq 1-\alpha\}$.

$\M^*$ and $\hat \M$ induce partitions of $B(\M^*, r^*)$ and $B(\hat \M, \hat r^*)$ respectively, into $k$--clusters, by defining, for $i=1,\dots,k$,
\begin{equation}\label{grupos}
\text{Cluster}  \ C_i:= \big\{x \in B(\M^*, r^*):  \Vert x - m^*_i\Vert \leq  \min_{j \neq i} \Vert x - m^*_j \Vert \big\},
\end{equation}
\begin{equation} \label{gruposempiricos}
\text{Cluster} \  \hat C_i:= \big\{x \in B(\hat \M, \hat r):  \Vert x - \hat m_i\Vert \leq \min_{j \neq i} \Vert x - \hat m_j \Vert \big\}.
\end{equation}
The points at a boundary between clusters can be assigned arbitrarily.
A functional version of ITkM can be found in \cite{caf:06}.
With this in hand, the fusion step of the RFM is done by applying ITkM to the set of the $k\times m$ centres. The whole algorithm is summarized in Table \ref{RFM-cluster}.

\begin{table}[h!]
	\begin{tabular}{|l|} \hline
		1) Split the sample into $m$ subsamples (recall that $n=ml$).
		\\
		2) To each subsample, apply the empirical version of $\alpha$-ITkM with $\alpha=\alpha_1$ and \\
		\hspace{0.5cm}obtain $\hat \M_{1}, \ldots, \hat \M_{m}$, each one with $k$ points in $\mathbb R^d$.\\
		3) Apply  the empirical version of $\alpha$-ITkM with $\alpha=\alpha_2$ to the set $\cup_{i=1}^m \hat \M_{i}$.
		\\
		4) Obtain the output of the algorithm $(\hat \M_{RFM}, \hat r_{RMF})$.\\
		5) Build the clusters by applying \eqref{gruposempiricos}.
		\\ \hline
	\end{tabular}
	\caption{RFM algorithm for clustering}
	\label{RFM-cluster}
\end{table}

\section{Simulation results}\label{sec:sim}
We now describe the simulations done with the RFM for the three applications described in the previous sections.
As the design of each simulation is specific to its application, we describe them separately.\\
All the simulations were done using an 8-core PC, Intel core i7-3770 CPU, 8GB of RAM, 64 bit  processor, with the R software package v. 3.3.0 running under Ubuntu.

\subsection{ Location and scatter matrix for finite dimensional spaces}

We use the same simulations to analyse both the location of the parameters and their scatter matrix. For the robust estimator we have applied the function CovMest in the R-package rrcov with the parameters given by default.\\ 
We draw samples from a centred $5$-dimensional Gaussian  distribution with a covariance matrix with all its off-diagonal elements equal to $0.2$.
For the outliers we use a $5$-dimensional Cauchy distribution with independent coordinates centred at $50$.
 We test two contamination levels, $p=0.2$ and $p=0.4$. 
We vary  the sample size $n$ within the set $\{0.1E6, 5E6,10E6\}$ and the number of subsamples $m\in \{100,500,1000,10000 \}$.
We replicate each simulation case $K=5$ times and report the average.
The estimators obtained by the RFM are the values which maximize the depth functions given in Eqs \eqref{dephsloc} and \eqref{dephscat} for the location and the scatter matrix respectively.
In each case, the maximization is done over the set of the $m$ estimates obtained from the subsamples.

The mean squared error (averaged over 5 replicates) for the location problem are given in Table \ref{location}.
 The estimators considered are the following: the average of the whole sample (MLE), the average of the robust location estimators (avROB),  the average of the $40\%$ deepest robust estimators (RFM1), and the deepest robust estimator (RFM).

\begin{table}[h!]
	\footnotesize
	\caption{Location Estimators for $p=0.2$ and $p=0.4$. \label{location}} 
\begin{center}
\begin{tabular}{rr|rrrr|rrrr}
\hline
\multicolumn{1}{c}{}&\multicolumn{1}{c}{}&\multicolumn{1}{c}{MLE}&\multicolumn{1}{c}{avROB}&\multicolumn{1}{c}{RFM1}& RFM&\multicolumn{1}{c}{MLE}&\multicolumn{1}{c}{avROB}&\multicolumn{1}{c}{RFM1}&\multicolumn{1}{c}{RFM}\tabularnewline
\hline
\multicolumn{1}{c}{$n$}&\multicolumn{1}{c}{$m$}&\multicolumn{4}{|c|}{$p=0.2$}&\multicolumn{4}{c}{$p=0.4$}\\
\hline
$ 0.1$&$  100$&$31.3$&$0.0098$&$0.0124$&$0.0297$&$44.2$&$0.0042$&$0.0076$&$0.0288$\tabularnewline
$ 0.1$&$  500$&$31.3$&$0.0097$&$0.0112$&$0.0426$&$44.2$&$0.1070$&$0.0081$&$0.0427$\tabularnewline
$ 0.1$&$ 1000$&$31.3$&$0.0097$&$0.0109$&$0.0477$&$44.2$&$1.3400$&$0.0231$&$0.0416$\tabularnewline
$ 1.0$&$  100$&$21.4$&$0.0021$&$0.0029$&$0.0074$&$44.1$&$0.0038$&$0.0045$&$0.0087$\tabularnewline
$ 1.0$&$  500$&$21.4$&$0.0021$&$0.0037$&$0.0110$&$44.1$&$0.0038$&$0.0038$&$0.0164$\tabularnewline
$ 1.0$&$ 1000$&$21.4$&$0.0021$&$0.0030$&$0.0159$&$44.1$&$0.0038$&$0.0053$&$0.0186$\tabularnewline
$ 1.0$&$10000$&$21.4$&$0.0022$&$0.0035$&$0.0261$&$44.1$&$1.3900$&$0.0186$&$0.0375$\tabularnewline
$ 5.0$&$  100$&$22.0$&$0.0009$&$0.0014$&$0.0032$&$45.9$&$0.0007$&$0.0014$&$0.0044$\tabularnewline
$ 5.0$&$  500$&$22.0$&$0.0009$&$0.0010$&$0.0056$&$45.9$&$0.0007$&$0.0014$&$0.0073$\tabularnewline
$ 5.0$&$ 1000$&$22.0$&$0.0009$&$0.0014$&$0.0071$&$45.9$&$0.0007$&$0.0014$&$0.0097$\tabularnewline
$ 5.0$&$10000$&$22.0$&$0.0009$&$0.0015$&$0.0147$&$45.9$&$0.0013$&$0.0011$&$0.0159$\tabularnewline
$10.0$&$  100$&$23.5$&$0.0009$&$0.0013$&$0.0026$&$47.0$&$0.0005$&$0.0010$&$0.0033$\tabularnewline
$10.0$&$  500$&$23.5$&$0.0009$&$0.0012$&$0.0038$&$47.0$&$0.0005$&$0.0009$&$0.0052$\tabularnewline
$10.0$&$ 1000$&$23.5$&$0.0009$&$0.0012$&$0.0047$&$47.0$&$0.0005$&$0.0008$&$0.0056$\tabularnewline
$10.0$&$10000$&$23.5$&$0.0009$&$0.0012$&$0.0090$&$47.0$&$0.0005$&$0.0009$&$0.0102$\tabularnewline
\hline
\end{tabular}\end{center}
\end{table}

We can see that the estimator obtained by the RFM behaves very well.
Depending on the structure of the outliers, the mean of the robust estimates may behave well or not.
Even if only one of the subsamples contains a high proportion of outliers causing the robust estimator to break down, the average of the robust estimators will break down.
On the other hand, the deepest $M$-estimator always behaves well.
 The performances of both estimators decrease in general with $m$.

The estimation errors for the covariance are given in Table \ref{signif1} ($p=0.2$) and Table \ref{signif2} ($p=0.4$).
We compare the MLE estimator (MLE), a robust estimator based on the whole sample (ROB), the average of the robust scatter matrix estimators (avROB),  the average of the $40\%$ deepest robust estimators (RFM1), and the deepest robust estimator (RFM). We also report the average time in seconds necessary for both the global estimator (T0, over the whole sample), and T1, the estimator obtained by fusion (including computing the estimators over subsamples and aggregating them by fusion). Since the second step of the algorithm (see point b) in Table \ref{RFM-alg})  can be parallelized,  in practice the computational time T1 can be divided almost by $m$. The results of RFM are very good for the covariance matrix as well.

\begin{table}[h!]
	\footnotesize
	\caption{Covariance estimators. Using MLE and robust estimates over the entire sample, and aggregating by average, trimmed average or fusion of $m$ subsamples estimators, $p=0.2$.} \label{signif1}
\begin{center}
\begin{tabular}{rrrrrrrrr}
\hline
\multicolumn{1}{c}{$n$}&\multicolumn{1}{c}{$m$}&\multicolumn{1}{c}{T0}&\multicolumn{1}{c}{T1}&\multicolumn{1}{c}{MLE}&\multicolumn{1}{c}{ROB}&\multicolumn{1}{c}{avROB}&\multicolumn{1}{c}{RFM1}&\multicolumn{1}{c}{RFM}\tabularnewline
\hline
$ 0.1$&$  100$&$ 0.460$&$   4.205$&$23688000$&$0.2594$&$0.2597$&$0.2598$&$0.3722$\tabularnewline
$ 0.1$&$  500$&$ 0.460$&$  14.527$&$23688000$&$0.2594$&$0.2675$&$0.2498$&$0.4810$\tabularnewline
$ 0.1$&$ 1000$&$ 0.460$&$  23.992$&$23688000$&$0.2594$&$0.2748$&$0.2418$&$0.6130$\tabularnewline
$ 1.0$&$  100$&$ 3.444$&$   6.524$&$ 1617200$&$0.2342$&$0.2345$&$0.2368$&$0.2656$\tabularnewline
$ 1.0$&$  500$&$ 3.444$&$  24.028$&$ 1617200$&$0.2342$&$0.2353$&$0.2371$&$0.3189$\tabularnewline
$ 1.0$&$ 1000$&$ 3.444$&$  45.307$&$ 1617200$&$0.2342$&$0.2360$&$0.2381$&$0.3295$\tabularnewline
$ 1.0$&$10000$&$ 3.444$&$ 945.350$&$ 1617200$&$0.2342$&$0.2464$&$0.2075$&$0.4982$\tabularnewline
$ 5.0$&$  100$&$20.528$&$  15.984$&$ 1981900$&$0.2317$&$0.2316$&$0.2340$&$0.2495$\tabularnewline
$ 5.0$&$  500$&$20.528$&$  33.289$&$ 1981900$&$0.2317$&$0.2316$&$0.2331$&$0.2687$\tabularnewline
$ 5.0$&$ 1000$&$20.528$&$  68.342$&$ 1981900$&$0.2317$&$0.2318$&$0.2336$&$0.2842$\tabularnewline
$ 5.0$&$10000$&$20.528$&$1312.800$&$ 1981900$&$0.2317$&$0.2342$&$0.2267$&$0.3810$\tabularnewline
$10.0$&$  100$&$42.174$&$  29.168$&$28135000$&$0.2307$&$0.2306$&$0.2322$&$0.2445$\tabularnewline
$10.0$&$  500$&$42.174$&$  49.992$&$28135000$&$0.2307$&$0.2307$&$0.2322$&$0.2567$\tabularnewline
$10.0$&$ 1000$&$42.174$&$  73.701$&$28135000$&$0.2307$&$0.2308$&$0.2315$&$0.2660$\tabularnewline
$10.0$&$10000$&$42.174$&$1291.000$&$28135000$&$0.2307$&$0.2323$&$0.2290$&$0.3439$\tabularnewline
\hline
\end{tabular}\end{center}
\end{table}

\begin{table}[!tbp]
	\footnotesize
	\caption{Covariance estimators. Using MLE and robust estimates over the entire sample, and aggregating by average, trimmed average or fusion the $m$ subsamples estimators, $p=0.4$.\label{signif2}} 
\begin{center}
\begin{tabular}{rrrrrrrrr}
\hline
$n$&$m$&\multicolumn{1}{c}{T0}&\multicolumn{1}{c}{T1}&\multicolumn{1}{c}{MLE}&\multicolumn{1}{c}{ROB}&\multicolumn{1}{c}{avROB}&\multicolumn{1}{c}{RFM1}&\multicolumn{1}{c}{RFM}\tabularnewline
\hline
$ 0.1$&$  100$&$ 0.581$&$   3.416$&$  448210$&$0.8247$&$  0.8348$&$0.8378$&$1.0111$\tabularnewline
$ 0.1$&$  500$&$ 0.581$&$  13.614$&$  448210$&$0.8247$&$ 16.3120$&$0.8057$&$1.2548$\tabularnewline
$ 0.1$&$ 1000$&$ 0.581$&$  22.827$&$  448210$&$0.8247$&$205.3000$&$0.7772$&$1.5159$\tabularnewline
$ 1.0$&$  100$&$ 2.631$&$   5.622$&$ 6030100$&$0.8081$&$  0.8094$&$0.8114$&$0.8790$\tabularnewline
$ 1.0$&$  500$&$ 2.631$&$  21.131$&$ 6030100$&$0.8081$&$  0.8143$&$0.8103$&$0.9462$\tabularnewline
$ 1.0$&$ 1000$&$ 2.631$&$  39.752$&$ 6030100$&$0.8081$&$  0.8201$&$0.8059$&$0.9690$\tabularnewline
$ 1.0$&$10000$&$ 2.631$&$ 833.530$&$ 6030100$&$0.8081$&$203.9100$&$0.7706$&$1.2760$\tabularnewline
$ 5.0$&$  100$&$16.651$&$  14.126$&$29809000$&$0.8010$&$  0.8012$&$0.8035$&$0.8299$\tabularnewline
$ 5.0$&$  500$&$16.651$&$  30.762$&$29809000$&$0.8010$&$  0.8021$&$0.8025$&$0.8571$\tabularnewline
$ 5.0$&$ 1000$&$16.651$&$  60.389$&$29809000$&$0.8010$&$  0.8032$&$0.8020$&$0.8740$\tabularnewline
$ 5.0$&$10000$&$16.651$&$1239.300$&$29809000$&$0.8010$&$  0.9024$&$0.7877$&$1.0311$\tabularnewline
$10.0$&$  100$&$33.922$&$  24.757$&$93071000$&$0.7988$&$  0.7989$&$0.8013$&$0.8185$\tabularnewline
$10.0$&$  500$&$33.922$&$  43.999$&$93071000$&$0.7988$&$  0.7993$&$0.8007$&$0.8420$\tabularnewline
$10.0$&$ 1000$&$33.922$&$  68.787$&$93071000$&$0.7988$&$  0.8001$&$0.8001$&$0.8555$\tabularnewline
$10.0$&$10000$&$33.922$&$1486.100$&$93071000$&$0.7988$&$  0.8117$&$0.7939$&$0.9403$\tabularnewline
\hline
\end{tabular}\end{center}
\end{table}

\subsection{ Covariance operator} \label{covfun}

\begin{figure}[h!]
	\begin{center}	
		\includegraphics[scale=0.3]{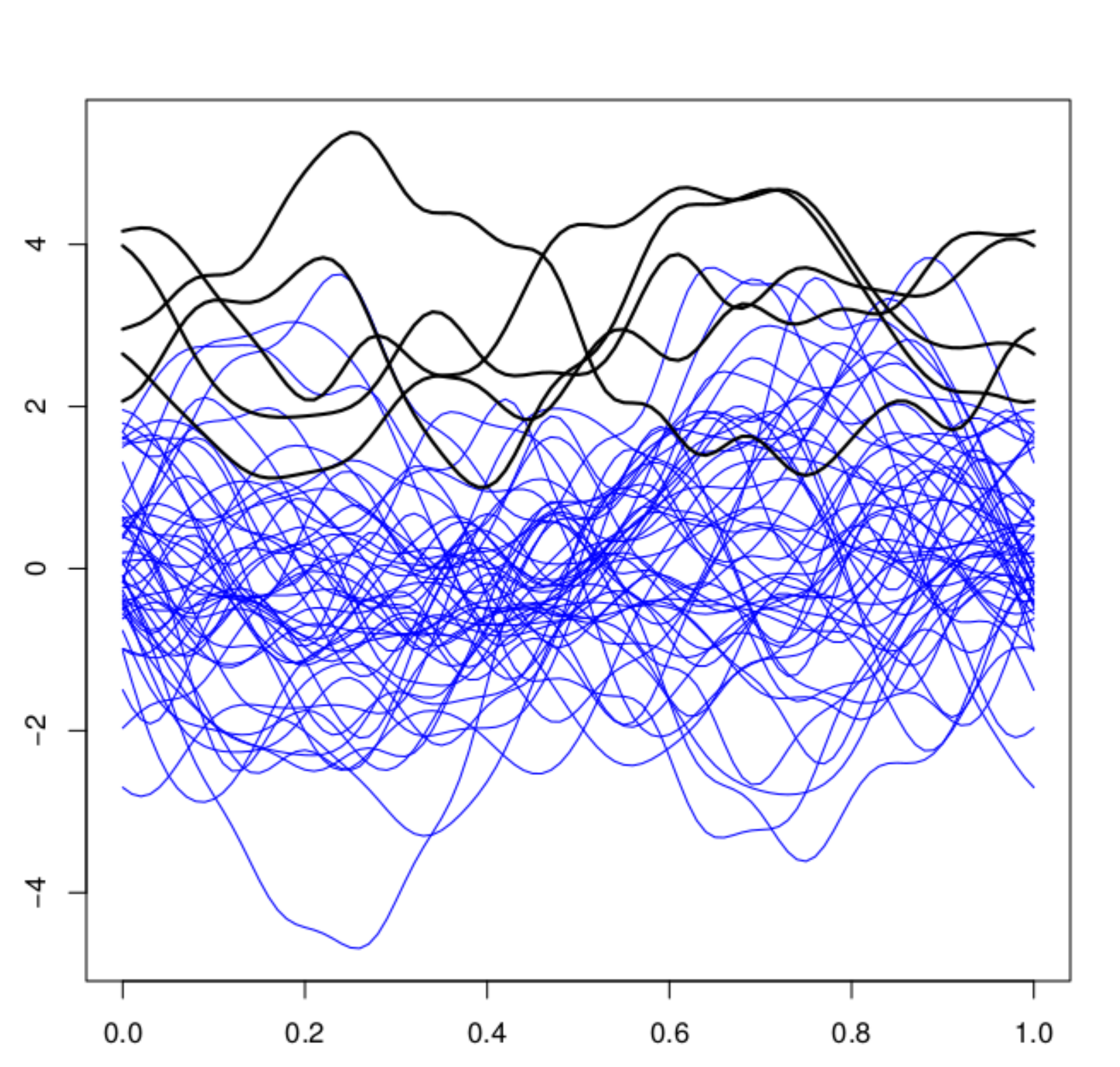}
	\end{center}
	\caption{Simulated functions and outliers}
	\label{}
\end{figure}

To generate the data, we have used a simplified version of  the simulation model used in \cite{Kraus}:
$$ X(t) = \mu(t) + \sqrt 2 \sum_{k=1}^{10} \lambda_k a_k \sin(2 \pi k t) + \sqrt 2 \sum_{k=1}^{10} \nu_k b_k \cos(2 \pi k t),$$
where $\nu_k = \left(\frac{1}{3}\right) ^k, \lambda_k = k^{-3}$, and $a_k$ and $b_k$ are random standard Gaussian independent observations.
The central observations were generated using $\mu(t) = 0$ whereas for the outliers we took $\mu(t) =2- 8\sin(\pi t)$.
For $t$ we used an equally spaced grid of $\mathcal{T}=20$ points in $[0,1]$.\\
The covariance operator of this process, given by $ Cov(s,t)= \sum_{k=1}^{10} A_k(s)A_k(t) + B_k(s)B_k(t) $, 
where $A_k(t) = \sqrt 2 \lambda_k \sin(2 \pi k t)$ and $ A_k(t) = \sqrt 2 \nu_k \cos(2 \pi k t)$, was computed for the comparisons.

We varied  the sample size $n$ within the set $\{0.1E6,1E6,5E6,10E6\}$ and the number of subsamples $m\in \{100,500,1000,10000 \}$.
The proportion of outliers was fixed to $p=0.15$ and $p=0.20$.
We replicated each simulation case $K=5$ times and report the average performance over the replicates.

 We report also the average time in seconds necessary for both a global estimate T0, over the whole sample, and T1, the estimate obtained by fusion (including computing the estimates over subsamples and aggregating them by fusion).
\\
We compare the classical estimator (MLE), the global robust estimate (ROB), the average of the robust estimates from the subsamples (avROB) and the robust fusion estimate (RFM).

The results are shown in Tables \ref{covfun1} and \ref{covfun2} for two proportions of outliers, $p=0.15$ and $p=0.2$ respectively.
\begin{table}[h!]
	\footnotesize
	\caption{Covariance operator estimator. Using the classical and robust estimators over the entire sample, and aggregating by average or fusion of $m$ subsamples estimators. $p=0.15$, $\mathcal{T}=20$.}\label{covfun1} 
	\begin{center}
		\begin{tabular}{rrrrrrrr}
			\hline
			\multicolumn{1}{c}{$n$}&\multicolumn{1}{c}{$m$}&\multicolumn{1}{c}{T0}&\multicolumn{1}{c}{T1}&\multicolumn{1}{c}{MLE}&\multicolumn{1}{c}{ROB}&\multicolumn{1}{c}{avROB}&\multicolumn{1}{c}{RFM}\tabularnewline
			\hline
			$0.05$&$  20$&$ 553$&$18.20$&$24.3$&$5.16$&$5.21$&$5.52$\tabularnewline
			$0.05$&$  50$&$ 543$&$ 7.81$&$24.3$&$5.20$&$5.24$&$5.60$\tabularnewline
			$0.05$&$ 100$&$ 528$&$ 4.79$&$24.3$&$5.20$&$5.17$&$5.58$\tabularnewline
			$0.05$&$1000$&$ 459$&$19.40$&$24.3$&$5.13$&$5.54$&$6.58$\tabularnewline
			$0.10$&$  20$&$2300$&$69.00$&$24.2$&$5.14$&$5.22$&$5.43$\tabularnewline
			$0.10$&$  50$&$2300$&$28.10$&$24.2$&$5.04$&$5.09$&$5.13$\tabularnewline
			$0.10$&$ 100$&$2290$&$15.20$&$24.2$&$5.06$&$5.15$&$5.43$\tabularnewline
			$0.10$&$1000$&$1850$&$21.60$&$24.3$&$5.21$&$5.35$&$6.13$\tabularnewline
			\hline
	\end{tabular}\end{center}	
\end{table}\vspace{-1cm}
\begin{table}[h!]
	\footnotesize
	\caption{Covariance operator estimator.  Using classical and robust estimators over the entire sample, and aggregating by average or fusion of $m$ subsamples estimators. $p=0.2$, $\mathcal{T}=20$.\label{covfun2}} 
	\begin{center}
		\begin{tabular}{rrrrrrrrrr}
			\hline
			\multicolumn{1}{c}{$n$}&\multicolumn{1}{c}{$m$}&\multicolumn{1}{c}{T0}&\multicolumn{1}{c}{T1}&\multicolumn{1}{c}{MLE}&\multicolumn{1}{c}{cvRob}&\multicolumn{1}{c}{avROB}&\multicolumn{1}{c}{RFM}\tabularnewline
			\hline
			$0.05$&$  20$&$ 572$&$17.90$&$30.5$&$0.879$&$ 3.96$&$1.45$\tabularnewline
			$0.05$&$  50$&$ 649$&$ 7.88$&$30.5$&$0.876$&$ 7.34$&$2.10$\tabularnewline
			$0.05$&$ 100$&$ 633$&$ 4.61$&$30.5$&$0.839$&$ 8.86$&$2.43$\tabularnewline
			$0.05$&$1000$&$ 478$&$19.50$&$30.5$&$0.864$&$13.10$&$7.08$\tabularnewline
			$0.10$&$  20$&$1970$&$69.10$&$30.4$&$0.914$&$ 3.83$&$1.36$\tabularnewline
			$0.10$&$  50$&$2030$&$28.10$&$30.4$&$0.921$&$ 4.32$&$1.55$\tabularnewline
			$0.10$&$ 100$&$2020$&$15.10$&$30.4$&$0.840$&$ 8.44$&$2.35$\tabularnewline
			$0.10$&$1000$&$1840$&$21.60$&$30.4$&$0.961$&$12.10$&$5.20$\tabularnewline
			\hline
		\end{tabular}
	\end{center}
\end{table}

\newpage

If the proportion of outliers is moderate, $p=0.15$, the average of the robust estimators still behaves well, better than RFM, but if we increase the proportion of outliers to $p=0.2$, RFM clearly outperforms all the other estimators.

\subsection{Clustering}

We performed a simulation study for large sample sizes, using a model with three clusters  with outliers, introduced in \cite{cagm:97}.
The data were generated using bivariate Gaussian distributions with the following parameters for the clusters and the outliers respectively:
	$$\mu_1 = (0,0),\;\mu_2 =(0,10) ,\;\mu_3 = (6,0),\;\mu_4 = (2,10/3),\;$$
	$$\Sigma_1 = \Sigma_2 = \Sigma_3 = 1.5\times Id,\;\Sigma_4 =20\times Id$$
	where $Id$ is the two dimensional identity matrix.
The outliers were generated with $\mu_4, \Sigma_4$ and $n_4$.
 The sizes of the clusters were fixed at the following values: $n_1 = 15 ,\;n_2 =30 ,\; n_3= 30,\;n_4 = 40.$
	As in \cite{cagm:97},  the outliers lying in the $75\%$ level confidence ellipsoids of the clusters were replaced by others not belonging to that area.
The outliers represent almost $35\%$ of the whole sample.
We used this base simulation and varied the whole sample size, multiplying each $n_i$ by a factor ``fac'' taking the values in $\{10,100,1000, 10000\}$.
So for the smallest sample, we have $n=1150$, and the largest, $n=1150000$.%
	\begin{figure}[h!]
	\begin{center} 
		\includegraphics[width=11.5cm,height=6cm]{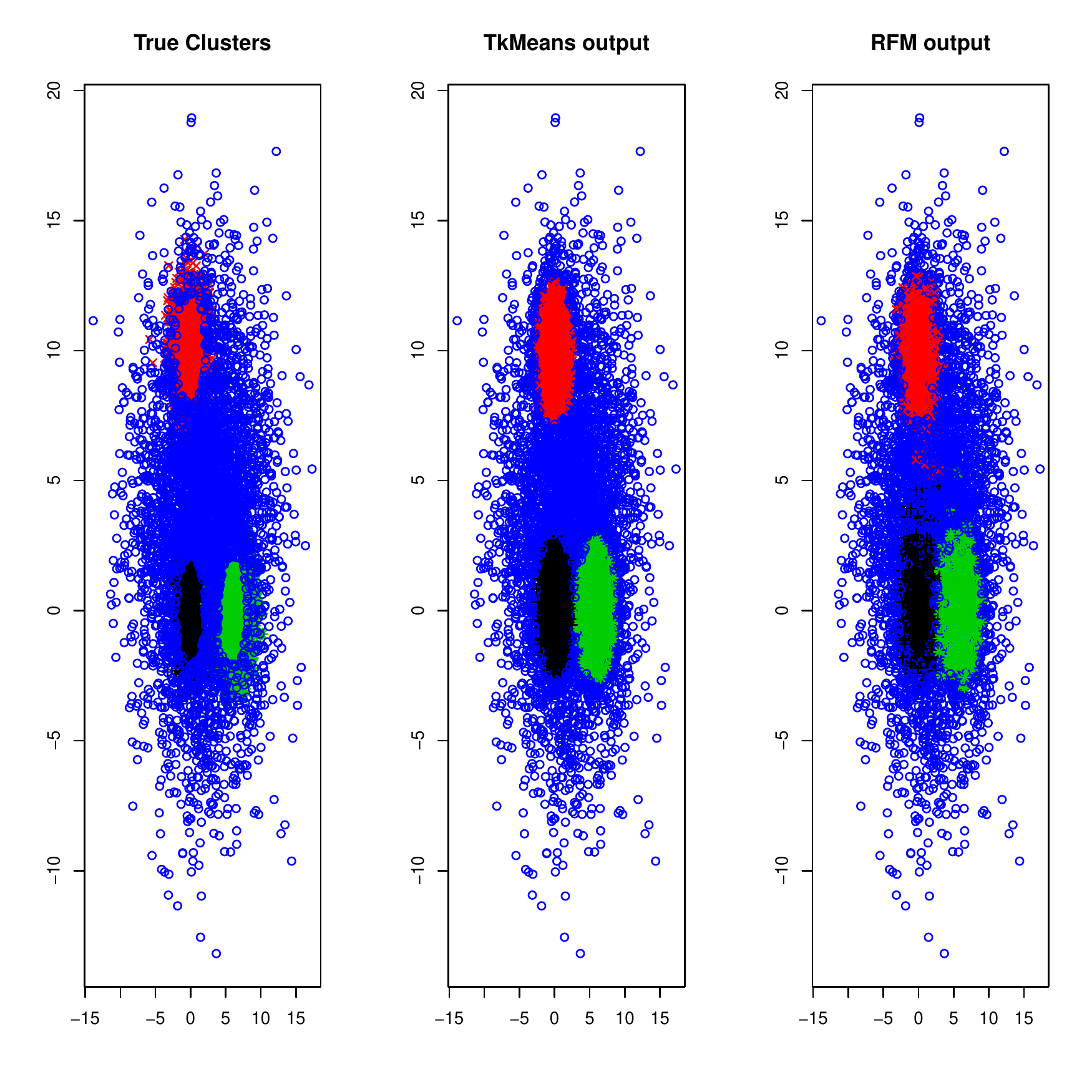}
	\end{center}
	\caption{Left panel: the true clusters. Middle panel: Results obtained by ITkM over the whole sample. Right panel: The output obtained by RFM using $m=100$ subsamples. The outliers are the blue points and $n=11500$.}
	\label{Clusters}
\end{figure}

	For each value of $n$ we varied the number of subsamples $m$ within the values $\{10,50,100,1000, 10000\}$ with the restriction $m<\text{fac}$.
Lastly, when applying the trimmed $k$-means to the samples, we have tested three values for the trimming level, $\alpha_1=0.2,0.35, 0.45$, whereas for the fusion we fixed $\alpha_2=0.1$.\\
	The left hand panel of Figure \ref{Clusters} shows an example of the simulated data-set for $n=11500$, the middle panel shows the results obtained by ITkM applied to the whole sample, and the right hand panel shows the output of the algorithm.\\
The partitions obtained by each approach are compared to the true clusters using the matching error defined by
	\begin{equation}\label{err}
	ME = \min_{s\in \mathcal{S}(k+1)} \ \ \frac{1}{n}  \sum_{i=1}^n \mathbf{1}_{\{y_i \neq s(\hat{y}_i)\}}
	\end{equation}
	where $\mathcal{S}(k+1)$ is the set of permutations of $\{0,1,\dots,k+1\}$, $y_i$ is the true cluster of  observation $i$ and $\hat {y}_i$ is the cluster assigned by the algorithm.	
	The results of the simulation are given in Table \ref{clus20}, where we compare the RFM method, with the ITkM calculated with the whole sample. Columns ME1 and ME2 give the matching errors for ITkM applied to the whole sample and for RFM respectively.
	We also report the average time in seconds necessary for both the global estimator (T0, over the whole sample), and T1, the estimator obtained by fusion (including computing the estimators over subsamples and aggregating them by fusion). Finally T2 is the time using parallel computing.
	
	 As expected, the RFM matching errors are often higher than those of ITkM applied to the whole sample. But the loss of performance is very small in general and increases with $m$. For the smallest values of $m$ with large samples ($n>10000$), RFM has almost the same performance for all values of $\alpha$. On the other hand, increasing the value of $m$ reduces considerably the computation time of RFM.

	\begin{table}[h!]
		\footnotesize
		\caption{RFM for Clustering using different values of the trimming parameter $\alpha_1$.\label{clus20}} 
		\begin{center}
			\begin{tabular}{rrrrrrr}
				\hline
				\multicolumn{1}{c}{$n$}&\multicolumn{1}{c}{$m$}&\multicolumn{1}{c}{T0}&\multicolumn{1}{c}{T1}&\multicolumn{1}{c}{T2}&\multicolumn{1}{c}{ME1}&\multicolumn{1}{c}{ME2}\tabularnewline
				\hline
				\multicolumn{7}{c}{$\alpha_1= 0.2$} \tabularnewline \hline
				$   1150$&$   10$&$   2.89$&$   1.34$&$  0.55$&$0.1539$&$0.1678$\tabularnewline
				$  11500$&$   10$&$  21.20$&$  21.69$&$  6.83$&$0.1594$&$0.1603$\tabularnewline
				$  11500$&$  100$&$  21.20$&$  14.65$&$  4.24$&$0.1594$&$0.1693$\tabularnewline
				$ 115000$&$   10$&$ 274.90$&$ 263.80$&$ 75.11$&$0.1585$&$0.1585$\tabularnewline
				$ 115000$&$  100$&$ 274.90$&$ 218.10$&$ 56.44$&$0.1585$&$0.1591$\tabularnewline
				$ 115000$&$ 1000$&$ 274.90$&$ 141.50$&$ 37.51$&$0.1585$&$0.1693$\tabularnewline
				$1150000$&$   10$&$3452.00$&$3149.00$&$873.40$&$0.1582$&$0.1582$\tabularnewline
				$1150000$&$  100$&$3452.00$&$2609.00$&$680.70$&$0.1582$&$0.1583$\tabularnewline
				$1150000$&$ 1000$&$3452.00$&$2158.00$&$546.90$&$0.1582$&$0.1590$\tabularnewline
				$1150000$&$10000$&$3452.00$&$1434.00$&$374.70$&$0.1582$&$0.1689$\tabularnewline
				\hline
				\multicolumn{7}{c}{ $\alpha_1= 0.35$} \tabularnewline \hline
				$   1150$&$   10$&$   3.45$&$   1.43$&$   0.54$&$0.1287$&$0.1310$\tabularnewline
				$  11500$&$   10$&$  37.87$&$  33.38$&$   9.89$&$0.1037$&$0.1071$\tabularnewline
				$  11500$&$  100$&$  37.87$&$  15.30$&$   4.29$&$0.1037$&$0.1343$\tabularnewline
				$ 115000$&$   10$&$ 427.70$&$ 391.10$&$ 109.60$&$0.1049$&$0.1050$\tabularnewline
				$ 115000$&$  100$&$ 427.70$&$ 307.20$&$  85.70$&$0.1049$&$0.1071$\tabularnewline
				$ 115000$&$ 1000$&$ 427.70$&$ 137.70$&$  38.36$&$0.1049$&$0.1331$\tabularnewline
				$1150000$&$   10$&$4925.00$&$4284.00$&$1166.00$&$0.1052$&$0.1053$\tabularnewline
				$1150000$&$  100$&$4925.00$&$3660.00$&$ 928.20$&$0.1052$&$0.1055$\tabularnewline
				$1150000$&$ 1000$&$4925.00$&$3052.00$&$ 792.90$&$0.1052$&$0.1074$\tabularnewline
				$1150000$&$10000$&$4925.00$&$1397.00$&$ 372.20$&$0.1052$&$0.1336$\tabularnewline
				\hline
				\multicolumn{7}{c}{$\alpha_1= 0.45$} \tabularnewline \hline
				$   1150$&$   10$&$   2.72$&$   1.27$&$   0.52$&$0.1330$&$0.1567$\tabularnewline
				$  11500$&$   10$&$  55.58$&$  34.12$&$   9.80$&$0.1370$&$0.1403$\tabularnewline
				$  11500$&$  100$&$  55.58$&$  13.11$&$   3.65$&$0.1370$&$0.1723$\tabularnewline
				$ 115000$&$   10$&$ 698.90$&$ 586.60$&$ 170.40$&$0.1325$&$0.1330$\tabularnewline
				$ 115000$&$  100$&$ 698.90$&$ 323.90$&$  86.35$&$0.1325$&$0.1355$\tabularnewline
				$ 115000$&$ 1000$&$ 698.90$&$ 122.50$&$  33.53$&$0.1325$&$0.1729$\tabularnewline
				$1150000$&$   10$&$7190.00$&$7087.00$&$2115.00$&$0.1327$&$0.1328$\tabularnewline
				$1150000$&$  100$&$7190.00$&$5654.00$&$1508.00$&$0.1327$&$0.1330$\tabularnewline
				$1150000$&$ 1000$&$7190.00$&$3287.00$&$ 829.60$&$0.1327$&$0.1360$\tabularnewline
				$1150000$&$10000$&$7190.00$&$1258.00$&$ 328.10$&$0.1327$&$0.1726$\tabularnewline
				\hline
		\end{tabular}\end{center}
	\end{table}

\subsubsection{A real data example}

	As an example we have chosen the MNIST data-set of handwritten digits (see https://www.kaggle.com/c/digit-recognizer/data) to compare the performance of the RFM clustering algorithm with the same clustering procedure without splitting the sample (impartial trimmed $k$-means). The digits have been size-normalized and centred in a fixed-size image of 
	$28 \times 28$ pixels.
	
	 The data-set consist of a training sample $\{(X_1,Y_1),\dots,(X_n,Y_n)\}$ of $n=42000$ data, and a test sample of $10000$ data. As it is explained in the aforementioned link: ``this classic dataset of handwritten images has served as the basis for benchmarking classification algorithms". However, as we are interested in clustering we will use only the sample $\{X_1,\dots,X_n\}$, searching for $k=10$ groups. This is a very difficult task: if the labels are chosen at random the probability to get at least half of the 42000 data well identified is  extremely close to zero. We cluster the $42000$ data using both methods. 

 	The design is the same as for the previous simulations. On the one hand we cluster the whole sample using the impartial trimmed $k$-mean algorithm for $k=10$.  On the other hand we use the RFM clustering method given in Table \ref{RFM-cluster} for $m=10,100,500$ and $1000$, with $\alpha_1=0.05$ and $\alpha_1=0.1$.
		 The labels $Y_1,\dots,Y_n$ are only used to calculate the misclassification error rates ME1 and ME2 defined in \eqref{err}.
	
	The results are given for $\alpha_1=0.05$ and $\alpha_2=0.1$ in Table \ref{mnist01} left, and for $\alpha_1=0.1=\alpha_2$ in Table \ref{mnist01} right . They show that: (a) this clustering problem is very difficult (b) the relative efficiencies of the RFM clustering procedures for $\alpha_1=0.05$ are $5\%,2\%,9\%$ and $6\%$ while the computational times fall down drastically to $17\%, 5\%, 3\%$ and $3\%$, for $m=10,100,500$ and $1000$ respectively. For $\alpha_1=0.1$ the efficiencies are $8\%,7\%,9\%$ and $8\%$, the computational times fall down to $16\%,7\%,5\%$ and $4\%$ for $m=10,100,500$ and $1000$ respectively.
	
	\begin{table}[h!]
		\caption{Robust clustering, $\alpha_2=0.1$,  $\alpha_1= 0.05$ (left) and $\alpha_1=0.1$ (right) \label{mnist01}.} 
		\begin{center}
					\footnotesize
			\begin{tabular}{rrrrr}
				\hline
				\multicolumn{1}{c}{$m$}&\multicolumn{1}{c}{T0}&\multicolumn{1}{c}{T1}&\multicolumn{1}{c}{ME1}&\multicolumn{1}{c}{ME2}\tabularnewline
				\hline
				$  10$&$8560$&$1540$&$0.477$&$0.503$\tabularnewline
				$ 100$&$8560$&$ 503$&$0.477$&$0.486$\tabularnewline
				$ 500$&$8560$&$ 244$&$0.477$&$0.520$\tabularnewline
				$1000$&$8560$&$ 253$&$0.477$&$0.508$\tabularnewline
				\hline
		\end{tabular} \hspace{0.3cm}
	\begin{tabular}{rrrrr}
		\hline
		m&\multicolumn{1}{c}{T0}&\multicolumn{1}{c}{T1}&\multicolumn{1}{c}{ME1}&\multicolumn{1}{c}{ME2}\tabularnewline
		\hline
		$  10$&$9570$&$1500$&$0.492$&$0.530$\tabularnewline
		$ 100$&$9570$&$ 705$&$0.492$&$0.525$\tabularnewline
		$ 500$&$9570$&$ 445$&$0.492$&$0.536$\tabularnewline
		$1000$&$9570$&$ 417$&$0.492$&$0.532$\tabularnewline
		\hline
\end{tabular}\end{center}
\end{table}

\section{Concluding remarks}
	We have addressed some fundamental statistical problems in the context of Big Data, namely large samples, in the presence of outliers; location and covariance estimation, covariance operator estimation, and clustering. We have proposed a general robust approach, called the robust fusion method (RFM), and shown how it may be applied to these problems. The simulations gave very good results mainly for the last two problems. \\
	\indent Different statistical challenges go through these problems. Our approach may be adapted to any other task as soon as a robust efficient estimate is available for the corresponding problem.
	
	\begin{itemize}
	\item We have addressed one of the important problems in Big Data, namely when the size of the data-set is too large and one needs to split it into pieces.
	\item In this setup we think that robustness is mandatory.
	\item We have provided a general procedure, a robust fusion method, to deal with these problems.
The method is very general and can be applied to different statistical problems for high dimensional and functional data.
	\item Robust methods should be reasonably simple, in order to work with very large samples.
	\item We have provided a new robust method (RFM) to estimate the covariance operator in the functional data setting.
	\item As particular cases we considered the multivariate location problem, the scatter matrix, the covariance operator, and clustering methods.
We have illustrated through simulated examples the behaviour of RFM for all these problems for different (large) sample sizes.
	\end{itemize}

{\bf Acknowledgement. }To the constructive comments and criticisms from an associated editor and  two
	anonymous referees. For the last author this work has been partially supported by the ECOS project:  No. U14E02.

\end{document}